\documentclass[12pt,a4paper]{amsart}

\usepackage{amsfonts, amsmath, amssymb, amsthm, amscd, hyperref}

\usepackage{a4wide}

\newtheorem{thm}{Theorem}
\newtheorem*{thm*}{Theorem}
\newtheorem{lem}{Lemma}

\newtheorem{cor}[thm]{Corollary}

\theoremstyle{definition}
\newtheorem{defn}{Definition}

\theoremstyle{remark}

\DeclareMathOperator{\suppo}{supp}
\DeclareMathOperator{\hind}{hind}
\DeclareMathOperator{\conv}{conv}

\DeclareMathOperator{\bd}{bd}

\DeclareMathOperator{\dist}{dist}

\renewcommand{\int}{\mathop{\rm int}}

\renewcommand{\epsilon}{\varepsilon}

\begin{document}

\title{Theorems of Borsuk-Ulam type for flats and common transversals}
\author{R.N.~Karasev}

\thanks{This research was partially supported by the Dynasty Foundation.}

\email{r\_n\_karasev@mail.ru}
\address{
Roman Karasev, Dept. of Mathematics, Moscow Institute of Physics
and Technology, Institutskiy per. 9, Dolgoprudny, Russia 141700}

\begin{abstract}
In this paper some results on the topology of the space of $k$-flats in $\mathbb R^n$ are proved, similar to the Borsuk-Ulam theorem on coverings of sphere. Some corollaries on common transversals for families of compact sets in $\mathbb R^n$, and on measure partitions by hyperplanes, are deduced. 
\end{abstract}

\subjclass[2000]{52A35,55M30,55M35}
\keywords{Borsuk-Ulam theorem, common transversals, Helly-type theorems}

\maketitle

\section{Introduction}

Let us remind some classical results, which are generalized in this paper. One of the most important results is the Borsuk-Ulam theorem on coverings of the sphere~\cite{bor1933}.

\begin{thm*}[The Borsuk-Ulam theorem]
If the sphere $S^n$ is covered by a family of $n+1$ closed (or open) sets $X_1,\ldots,X_{n+1}$, then at least one of $X_i$ contains a pair of antipodal points of $S^n$.
\end{thm*}

Note that the sphere $S^n$ is considered to be a unit sphere in $\mathbb R^{n+1}$, and the points $x$ and $-x$ are called antipodal. This theorem can be reformulated for coverings of a ball.

\begin{thm*}[The Borsuk-Ulam theorem for coverings of the ball]
Let a ball $B^n\in\mathbb R^n$ be covered by closed (or open) sets $X_1,\ldots,X_{n+1}$, and for any $i=1,\ldots, n+1$ the intersection $X_i\cap\partial B^n$ does not contain a pair of antipodal points. Then the intersection $\bigcap_{i=1}^{n+1} X_i$ is non-empty.
\end{thm*}

In this paper we are going to consider the configuration spaces of $k$-flats ($k$-dimensional affine subspaces) in $\mathbb R^n$, and prove the generalizations of the Borsuk-Ulam theorem for coverings of such spaces. 

Let us state another famous theorem about intersections of convex sets in $\mathbb R^n$, the Helly theorem~\cite{helly1923}.

\begin{thm*}[Helly's theorem]
Let $\mathcal F$ be a finite family of convex sets in $\mathbb R^n$. The family $\mathcal F$ has a common points, iff any subfamily $\mathcal G\subseteq \mathcal F$ with size $|\mathcal G|\le n+1$ has a common point.
\end{thm*}

We are going to use the generalization of the Helly theorem from~\cite{bar1982}.

\begin{thm*}[The colored Helly theorem]
Let $\mathcal F_1, \ldots,\mathcal F_{n+1}$ be families of convex compact sets in $\mathbb R^n$. Suppose that for any system of representatives $\{X_i\in\mathcal F_i\}_{i=1}^{n+1}$ the intersection $\bigcap_{i=1}^{n+1} X_i$ is non-empty. Then for some $i$ the intersection $\bigcap\mathcal F_i$ is non-empty.
\end{thm*}

The Helly theorem gives a condition for existence of a common point for the whole family in terms of existence of a common point for subfamilies of given size. It also makes sense to find a $k$-flat that intersects every set in the family, and try to find some sufficient conditions on its existence. Let us make a definition.

\begin{defn}
If a $k$-flat $L$ intersects every set of the family $\mathcal F$, then $L$ is called \emph{(common) $k$-transversal} for the family $\mathcal F$.
\end{defn}

In fact, the existence of $k$-transversal for arbitrary finite family of convex sets cannot be deduced from any Helly-type theorem. Thus it makes sense to modify the conditions in Helly-type theorems to provide a common transversal, see~\cite{dol1993,cgppsw1994} for example. 

Let us state a result on common transversals from~\cite{horn1949,klee1951}, that generalizes the Helly theorem.

\begin{thm*}[The Horn-Klee theorem] 
Let $1\le k\le d$ be integers, let $\mathcal F$ be a family of convex compact sets in $\mathbb R^d$. Then the following conditions are equivalent:

1) Every $k$ or less sets of $\mathcal F$ have a common point;

2) Every flat of codimension $k-1$ in $\mathbb R^d$ can be translated to intersect every member of $\mathcal F$;

3) Every flat of codimension $k$ in $\mathbb R^d$ is contained in a flat of codimension $k-1$, that is a transversal for $\mathcal F$.
\end{thm*}

In Section~\ref{hellytype} we prove some results on common transversals, close to the colored Helly theorem and the Horn-Klee theorem. In Section~\ref{buplanes} we prove some Borsuk-Ulam-type theorems, that give sufficient conditions for existence of a common $k$-transversal for $n+1$-element families of (possibly non-convex) sets in $\mathbb R^n$.

Let us state the result on measure partitions from~\cite{st1942,ste1945}, that can be deduced from the Borsuk-Ulam theorem.

\begin{thm*}[The ``ham sandwich'' theorem]
Suppose that $d$ absolutely continuous probabilistic measures $\mu_1,\ldots,\mu_d$ are given in $\mathbb R^d$. Then there is a half-space $H\subset\mathbb R^d$ such that for any $i=1,\ldots,d$
$$
\mu_i(H) = 1/2.
$$
\end{thm*}

It is natural to ask, whether we can partition the measures into parts of arbitrary measure. More precisely, suppose we are given numbers $(\alpha_1, \ldots, \alpha_d) \in [0, 1]^d$, and try to find a half-space $H\subseteq \mathbb R^d$ such that for any $i=1,\ldots,d$ the measure is $\mu_i(H) = \alpha_i$. This cannot be done in general, it is sufficient to consider several uniform measures on concentric balls. 

Some additional conditions on the measures are required, in the papers~\cite{bi1990,cgppsw1994,klh1997,bhj2007,breu2009} it was shown that it is sufficient to require the supports of the measures to be separated, i.e. for any system of representatives $x_i\in\conv \suppo \mu_i$, the points $(x_1, \ldots, x_d)$ should be affine independent. In Section~\ref{mespart} we study the measure partitions by hyperplanes and prove some generalization of this result.

\section{Theorems on the canonical bundle over the Grassmannian}

In this section we study the configuration space of $k$-flats in $\mathbb R^n$ and obtain its properties, which are required in the study of common transversals and measure partitions. 

Denote the index set $[n] = \{1,2,\ldots, n\}$.

\begin{defn}
The subsets $X$ and $Y$ of a linear space $L$ are called \emph{separated}, if there exists a linear function (a polynomial of degree $1$) $l : L\to\mathbb R$ such that $l(X) < 0$ and $l(Y) > 0$.
\end{defn}

\begin{defn}
The families $\mathcal F$ and $\mathcal G$ of subsets of a linear space $L$ are called \emph{separated}, if there exists a linear function $l : L\to\mathbb R$ such that
$l(X) < 0$ for any $X\in\mathcal F$, and $l(Y) > 0$ for any $Y\in\mathcal G$.
\end{defn}

\begin{defn}
Two families of segments $\mathcal A$ and $\mathcal B$ in the real line are called \emph{equalized}, if one of the following alternatives holds:

1) All right ends of $\mathcal A$ coincide in $a$, all left ends of $\mathcal B$ coincide in $b$, either $a$ is to the left of $b$, or every segment of $\mathcal A\cup\mathcal B$ contains the segment $[ba]$;

2) All right ends of $\mathcal B$ coincide in $b$, all left ends of $\mathcal A$ coincide in $a$, either $b$ is to the left of $a$, or every segment of $\mathcal B\cup\mathcal A$ contains the segment $[ab]$.
\end{defn}

Denote $\gamma_n^k$ the canonical vector bundle over the Grassmannian $G_n^k$ of linear $k$-subspaces in $\mathbb R^n$.

In the sequel the continuous dependence of a convex compact set on some parameter is considered in the Hausdorff metric.

\begin{thm}
\label{hyperpl2} 
Consider $n+1$ closed subsets $V_1, V_2,\ldots,V_{n+1}$ in $\gamma_n^1$, such that the intersection of $V_i$ with any fiber $L$ is a non-empty segment (possibly one point), depending continuously on $L$. Then one of the alternatives hold:

1) There exists a fiber $L$ and $i\in [n+1]$ such that $V_i\cap L$ is contained in every $V_j\cap L$, for $j\not=i$;

2) For any partition of the family $\{V_i\}$ into non-empty subfamilies $\mathcal F_1$ and $\mathcal F_2$ there is a fiber $L$ such that the families of segments
$$
\mathcal F_1(L) = \{U\cap L : U\in\mathcal F_1\}\quad\text{and}\quad
\mathcal F_2(L) = \{U\cap L : U\in\mathcal F_2\}
$$
are equalized in $L$.
\end{thm}

It is clear that a pair of equalized families of segments is either separated, or has a common point. Thus Theorem~\ref{hyperpl2} implies the following.

\begin{cor}
\label{hyperpl1} 
Consider $n+1$ closed subsets $V_1, V_2,\ldots,V_{n+1}$ in $\gamma_n^1$, such that the intersection of $V_i$ with any fiber $L$ is a non-empty segment (possibly one point), depending continuously on $L$. Then either all the sets $V_i$ have a common point; or for any partition of the family $\{V_i\}$ into non-empty subfamilies $\mathcal F_1$ and $\mathcal F_2$ there is a fiber $L$ such that the families of segments
$$
\mathcal F_1(L) = \{U\cap L : U\in\mathcal F_1\}\quad\text{and}\quad
\mathcal F_2(L) = \{U\cap L : U\in\mathcal F_2\}
$$
are separated in $L$.
\end{cor}

Now we are going to state some results for arbitrary $k$ and the canonical bundle $\gamma_n^k\to G_n^k$. Let us make some definitions.

\begin{defn}
Let $K\subset\mathbb R^n$ be a convex compact set. A pair of points on $\partial K$ is called \emph{antipodal} w.r.t. $K$, if they can be enclosed into a pair of support hyperplanes of $K$ with opposite outer normals.
\end{defn}

In other words, the points $x$ and $y$ are antipodal w.r.t. $K$, iff the segment $[xy]$ is an affine diameter of $K$.

\begin{defn}
A family of compact sets $\mathcal F$ in $\mathbb R^n$ is called \emph{non-antipodal}, if none of the sets $V\in\mathcal F$ contains a pair of points, antipodal w.r.t. $\conv\bigcup\mathcal F$.
\end{defn}

In the sequel we assume that for any vector bundle we have some norm on the fibers, that has a smooth unit ball, and depends continuously on the fiber. In particular, we can consider the standard Euclidean norm on $\gamma_n^k$, or some other norm.

\begin{thm}
\label{kpl1} 
Consider $n+1$ compact sets $V_1, V_2,\ldots,V_{n+1}$ in $\gamma_n^k$ such that for any $i=1,\ldots, n+1$ the intersection with fiber $V_i\cap L$ is nonempty and depends continuously on $L$ in the Hausdorff metric. Suppose also that for any fiber $L\in G_n^k$ the family $\{V_i\cap L\}_{i=1}^{n+1}$ is non-antipodal in $L$. Then there exists a point $x$ in some fiber $L$ such that the distances from $x$ to all $V_i\cap L$ $(i=1,\ldots, n+1)$ are equal.
\end{thm}

\begin{thm}
\label{kpl2} 
Consider $n+1$ compact sets $V_1, V_2,\ldots,V_{n+1}$ in $\gamma_n^k$ such that for any $i=1,\ldots, n+1$ the intersection with fiber $V_i\cap L$ is nonempty and depends continuously on $L$ in the Hausdorff metric. Suppose also that for any fiber $L\in G_n^k$ the family $\{V_i\cap L\}_{i=1}^{n+1}$ is non-antipodal in $L$, and the union $\left(\bigcup_{i=1}^{n+1} V_i\right)\cap L$ is convex. Then the sets $V_i$ have a common point.
\end{thm}

The following theorem gives a partial solution to the conjecture on the fields of polytopes in the vector bundle $\gamma_n^k$ (see~\cite{mak2007}, Conjecture~1 and Theorem~12).

\begin{thm}
\label{polsection}
Consider $m$ continuous sections $s_1,\ldots, s_m$ of the bundle $\gamma_n^k$, such that for any fiber $L\in G_n^k$ the polytope $P(L) = \conv\{s_1(L), \ldots, s_m(L)\}$ has non-empty interior. Then there exists a fiber $L\in G_n^k$ and a pair of disjoint support half-spaces of $P(L)$, say $H_1, H_2\subset L$, such that the union $H_1\cup H_2$ contains at least $n+1$ points of $\{s_1(L), \ldots, s_m(L)\}$.
\end{thm}

In Sections~\ref{mespart}, \ref{buplanes}, \ref{hellytype} we deduce some corollaries from the above theorems, and prove some results using the similar technique.

\section{Some topological assertions}

Let us state some definitions of equivariant topology, see the book ~\cite{hsiang1975} for more detailed discussions.

\begin{defn}
Let $G$ be a compact Lie group or a finite group. A space $X$ with continuous action of $G$ is called a \emph{$G$-space}. A continuous map of $G$-spaces, commuting with the action of $G$ is called a \emph{$G$-map} or an \emph{equivariant map}. A $G$-space is called \emph{free} if the action of $G$ is free.
\end{defn}

There exists the universal free $G$-space $EG$ such that any other $G$-space maps uniquely (up to $G$-homotopy) to $EG$. The space $EG$ is homotopy trivial, the quotient space is denoted $BG = EG/G$. For any $G$-space $X$ and an Abelian group $A$ the equivariant cohomology $H_G^*(X, A)=H^*(X\times_G EG, A)$ is defined, and for free $G$-spaces the equality $H_G^*(X, A) = H^*(X/G, A)$ holds.

In this paper we consider the action of $G=Z_2$ only. Note that 
$$
H_G^*(\mathrm{pt}, Z_2) = H^*(\mathbb RP^\infty, Z_2) = Z_2[w]=\Lambda,
$$ 
where the dimension of the generator is $\dim w = 1$. Since any $G$-space $X$ can be mapped to the point $\pi_X: X\to \mathrm{pt}$, we have a natural map $\pi_X^* : \Lambda\to H_G^*(X, Z_2)$, the image $w$ under this map will be denoted $w$, if it does not make a confusion. The generator element of $Z_2$ will be denoted $\sigma$.

It is important to use the \v Cech or Alexander-Spanier cohomology to consider arbitrary closed subsets and their cohomology, because of their continuity property w.r.t. intersections.

\begin{defn}
The \emph{cohomology index} of a $Z_2$-space $X$ is the maximal $n$ such that the power $w^n\not=0$ in $H_G^*(X,Z_2)$. If there is no maximum, we consider the index equal to $\infty$. Denote the index of $X$ by $\hind X$.
\end{defn}

It is quite clear that the index of a non-free $Z_2$-space equals $\infty$. From the explicit description of the cohomology of $\mathbb RP^n$ it follows that the index of the $n$-dimensional sphere with the antipodal action of $Z_2$ ($x\mapsto -x$) equals $n$.

Let us state the following well-known lemma.

\begin{lem}[The generalized Borsuk-Ulam theorem for odd maps]
\label{index-bu}
If there exists an equivariant map $f: X\to Y$, then $\hind X\le \hind Y$.
\end{lem}

The lemma follows from the definition of index and the naturality of maps $\pi_X$ and $\pi_Y$. It is also called the monotonicity property of index.

Let us introduce some geometrical construction for a $Z_2$-space.

\begin{defn}
Let $X$ be a free $Z_2$-space. Take the product $X\times I$, where $I=[0,1]$ is the segment, and define the action of $Z_2$ by $\sigma(x, t) = (\sigma(x), 1-t)$. The space $X\times I$ is free, so we put $B(X) = (X\times I)/Z_2$. Informally, $B(x)$ is obtained from $X$ by gluing a segment to any pair $\{x, \sigma(x)\}$ and introducing the respective topology on the set of segments. The natural map $B(X)\to X/Z_2$ is a fiber bundle with fiber $I$, and a homotopy equivalence.
\end{defn}

\begin{defn}
Define the map $i_X:X\to B(X)$ by the formula $i_X(x)\mapsto (x, 0)$. It identifies $X$ with sphere bundle of the line bundle $B(X)\to X/Z_2$. In the sequel we always identify $X$ with a subset of $B(X)$ by the map $i_X$.
\end{defn}

Note that the space $B(X)$ has a natural $Z_2$-action, given by $(x, t)\mapsto (\sigma(x), t)$, with fixed-point set $(X\times \{1/2\})/Z_2$.

The following Theorem is a slight generalization of Lemma~5.5 from~\cite{volsce2005}.

\begin{thm}
\label{equimaps} Suppose that the $Z_2$-spaces $X$ and $Y$ have $\hind X = \hind Y = n$. Then for any equivariant map $f : X\to Y$ the map $f^*:H^n(Y,Z_2)\to H^n(X,Z_2)$ is nontrivial. Moreover, there does not exist a (continuous, non-equivariant) map $h: B(X)\to Y$ such that $f = h\circ i_X$.
\end{thm}

\begin{proof}
The spaces $X$ and $Y$ are obviously free. Consider the Thom exact sequences
$$
\ldots\xleftarrow{\delta_X} H^k(X, Z_2)\xleftarrow{i_X^*} H^k(B(X), Z_2)\xleftarrow{\pi_X^*} H^k(B(X), X, Z_2)\xleftarrow{\delta_X} H^{k-1}(X, Z_2)\ldots,
$$
$$
\ldots\xleftarrow{\delta_Y} H^k(Y, Z_2)\xleftarrow{i_Y^*} H^k(B(Y), Z_2)\xleftarrow{\pi_Y^*} H^k(B(Y), Y, Z_2)\xleftarrow{\delta_Y} H^{k-1}(Y, Z_2)\ldots.
$$
Note that $f$ gives a natural continuous map $B(X)\to B(Y)$, hence the above exact sequences are mapped into each other by $f^*$, which commutes with the exact sequence maps. Let the Thom classes in $H^1(B(X), X, Z_2)$ and $H^1(B(Y), Y, Z_2)$ be $u_X$ and $u_Y$; let the images of $w\in H_G^*(\mathrm{pt}, Z_2)$ in $H^*(B(X), Z_2)$ and $H^*(B(Y), Z_2)$ be $w_X$ and $w_Y$. It is clear that $f^*(u_Y) = u_X, f^*(w_Y) = w_X$. 

By the definition of index we have $\pi_Y^*(u_Xw_X^n) = w_X^{n+1}=0$ and $\pi_X^*(u_Yw_Y^n) = w_Y^{n+1}=0$, from the exactness there exists a class $v\in H^n(Y, Z_2)$ such that $\delta_Y(v) = u_Yw_Y^n$. We have $\delta_X(f^*(v)) = f^*(\delta_Y(v)) = u_Xw_X^n\neq 0$ and therefore $f^*(v)\neq 0$, the first claim of the theorem is proved.

To prove the second claim, note that the existence of $h$ would imply $f^*(v)\in\mathop{\mathrm{Im}} i_X^*$, and from the exactness of the Thom sequence $\delta_X(f^*(v)) = 0$, which contradicts the formulas in the previous paragraph.
\end{proof}

Now we are going to deduce a corollary from Theorem~\ref{equimaps} on coverings of $Z_2$-spaces. We need a definition first.

\begin{defn}
The points $x$ and $\sigma(x)$ of some $Z_2$-space $X$ are called \emph{antipodal}.
\end{defn}

The term ``antipodal'' was already defined for the points on the boundary of a convex compact set. But actually it does not lead to a confusion in the sequel.

\begin{thm}
\label{antipcover} 
Let a compact metric $Z_2$-space $X$ with $\hind X=n$ be covered by a family of closed sets $\mathcal F = \{U_1,U_2,\ldots,U_{n+2}\}$, so that none of $U_i$ contains a pair of antipodal points. Then for any partition of the family $\{U_i\}$ into non-empty subfamilies $\mathcal F_1$ and $\mathcal F_2$ there exists a point $x\in X$ such that 
$$
x\in\bigcap\mathcal F_1\quad\text{and}\quad \sigma(x)\in\bigcap\mathcal F_2.
$$ 
Moreover, if the covering $\mathcal F$ is induced by some closed covering $\mathcal G = \{V_1,V_2,\ldots,V_{n+2}\}$ of $B(X)$ then the family $\mathcal G$ has a common point.
\end{thm}

\begin{proof}

Suppose $W$ is a subset of some metric space $M$ and put
$$
W(\varepsilon) = \{x \in M : \dist(x, W) \le \varepsilon\}.
$$

From the compactness consideration we can assume that
$$
\exists\varepsilon>0 : \forall i=1,\ldots,n+2\ \dist(U_i,\sigma(U_i)) > 2\varepsilon.
$$
Consider a partition of unity $\{f_i\}_{i=1}^{n+2}$, subordinated to the covering by $U_i(\varepsilon)$, such that $f_i > 0$ on $U_i$.

Now consider the functions $g_i(x) = f_i(x) - f_i(\sigma(x))$. They together give an equivariant map $g : X\to \mathbb R^{n+2}$, the image of $g$ is contained in the hyperplane $H$, given by the equation $y_1+\ldots+y_{n+2} = 0$, and does not contain the origin. Thus the map $g$ induces the map $h : X\to S^n$ to the unit sphere of $H$, given by the formula $h(x) = g(x)/|g(x)|$.

The map $h : X\to S^n$ is equivariant and by Theorem~\ref{equimaps} the map $h^*: H^n(S^n, Z_2)\to H^n(X, Z_2)$ is nontrivial, hence $h$ must be surjective.

Take a partition $[n+2] = I_1\cup I_2$ and the corresponding partition 
$$
\{U_i(\varepsilon)\}_{i=1}^{n+2}=\mathcal F_1(\varepsilon)\cup\mathcal F_2(\varepsilon).
$$
Take the number $c = \sqrt{|I_1| |I_2| (n+2)}$ and consider a point $y\in S^n$ with coordinates
$$
y_i=\frac{|I_2|}{c}\ i\in I_1\quad y_i=-\frac{|I_1|}{c}\ i\in I_2.
$$

There exists a point $x\in X$ such that $y = h(x)$, hence $x\in\bigcap\mathcal F_1(\varepsilon)$, $\sigma(x)\in\bigcap\mathcal F_2(\varepsilon)$. Going to the limit with $\varepsilon\to 0$ and applying the compactness consideration we obtain the first claim of the theorem.

Now suppose that the covering is induced by a covering of $B(X)$, the functions $f_i$ give therefore a partition of unity on $B(X)$. They give the map $f$, that maps $X$ to the hyperplane $H_1$, given by the equation $y_1+\ldots+y_{n+2} = 1$. If the image of $f$ contains the point $(\frac{1}{n+2}, \ldots, \frac{1}{n+2})$ then, similar to the above reasoning the sets $\{V_i\}$ have a common point. Otherwise, $f$ gives a map $h_1 : B(X)\to S^n$ to the unit sphere of the hyperplane $H_1$. Let us show that the maps $h$ and $h_1|X$ are homotopy equivalent. Put for any $t\in [0, 1]$
$$
g_t(x) = f(x) - (1-t) f(\sigma(x))\quad h_t(x) = \frac{g_t(x) - t(\frac{1}{n+2}, \ldots, \frac{1}{n+2})}{|g_t(x) - t(\frac{1}{n+2}, \ldots, \frac{1}{n+2})|},
$$
then $h_t$ is the required homotopy. By the second claim of Theorem~\ref{equimaps} $h$ (and its homotopy equivalent $h_1$) cannot be extended from $X$ to $B(X)$ that is a contradiction.
\end{proof}

Let us prove another theorem on coverings.

\begin{thm}
\label{antipcover2} 
Let a compact metric $Z_2$-space $X$ with $\hind X=n$ be covered by a family of closed sets $\mathcal F = \{U_1,U_2,\ldots,U_N\}$. Suppose that none of $U_i$ contains a pair of antipodal points. Then there exists a point $x\in X$ such that the number of sets $U_i$, that contain either $x$ or $\sigma(x)$ is at least $n+2$.
\end{thm}

\begin{proof}
As in the previous theorem, consider the corresponding partition of unity $f_i: X\to \mathbb R^+$. Consider the map $g(x) = f(x) - f(\sigma(x))$ and assume the contrary, i.e. for any $x\in X$ at most $n+1$ of the coordinates of $g(x)$ are nonzero.

Since the sum of the positive coordinates $g_i(x)$ is $1$, the sum of negative coordinates is $-1$, the image of $g$ is contained in some $n-1$-dimensional simplicial complex. This complex has free $Z_2$-action, and its index is at most $n-1$ from the dimension considerations, thus we have a contradiction with Lemma~\ref{index-bu}.
\end{proof}

Let us state a theorem, that strengthens the Lyusternik-Schnirelmann theorem on the category of $\mathbb RP^n$.

\begin{defn}
An invariant subset $U\subseteq X$ of a $Z_2$-space $X$ is called \emph{inessential}, if $\hind U=0$.
\end{defn}

Note that the index of $U$ is zero, iff there exists an equivariant map of $U$ to the zero-dimensional sphere (a pair of points), here the continuity of cohomology is used.

\begin{thm}
\label{ls}
Let a compact metric $Z_2$-space $X$ with $\hind X=n$ be covered by a family of closed inessential invariant subsets $\mathcal F = \{U_1,U_2,\ldots,U_N\}$. Then $N\ge n+1$ and some $n+1$ sets of the family $\mathcal F$ have a common point.
\end{thm}

\begin{proof}
For any $U_i$ we have a $Z_2$-equivariant map $f_i: U_i\to \{-1, +1\}$. Extend $f_i$ to an equivariant map $f_i : X\to [-1, +1]$ so that $f_i$ is zero outside some $\varepsilon$-neighborhood of $U_i$. Finally, the functions $f_i$ give an equivariant map $f: X\to \mathbb R^N\setminus\{0\}$, hence an equivariant map $h : X\to S^{N-1}$ can be defined by $h(x) = f(x)/|f(x)|$. By Lemma~\ref{index-bu} $n\le N-1$ and the first claim is proved.

To prove the second claim assume the contrary: for sufficiently small $\varepsilon$ it would mean that at most $n$ of the coordinates of $h(x)$ can be nonzero. Thus the image if $h$ is contained in some $n-1$-dimensional subset of $S^{N-1}$, that contradicts with Lemma~\ref{index-bu}.
\end{proof}

Another analogue of Theorem~\ref{antipcover} can be proved for a product of $Z_2$-spaces (see Theorem~2 from~\cite{kar2007}).

\begin{thm}
\label{antipcover3} 
Let compact metric $Z_2$-spaces $X$ and $Y$ have indexes $\hind X= n,\ \hind Y= m,\ m\ge n$.

Let the set $B(X)\times B(Y)$ be covered by a family of closed subsets $\mathcal F = \{V_{i j}\}_{i\in[n+2], j\in [m+2]}$. Denote the projections $\pi_X : B(X)\times B(Y)\to B(X),\ \pi_Y B(X)\times B(Y)\to B(Y)$.

Suppose that for any $i\in[n+2]$ the set $\pi_X(\bigcup_{j\in [m+2]}V_{i j})\cap X$ does not contain a pair of antipodal points, and for any $j\in[m+2]$ the set $\pi_Y(\bigcup_{i\in [n+2]}V_{i j})\cap Y$ does not contain a pair of antipodal points either. 

Let $\{a_i\}_{i\in [n+2]}$ be positive integers with sum equal to $m+2$. Then there exists a map  $\tau : [m+2]\to [n+2]$ such that
$$
\forall i\in[n+2]\ |\tau^{-1}(i)| = a_i\quad\text{and}\quad \bigcap_{j\in[m+2]} U_{\tau(j) j}\neq\emptyset.
$$
\end{thm}

We need the generalized Hall theorem~\cite{harary1994} on matchings. Here we denote the positive integers $\mathbb N = \{i\in \mathbb Z : i \ge 1 \}$.

\begin{lem}
\label{hall}
Consider a bipartite graph with vertices $V\cup W$, where $|V| = n$, $|W|=m$ $(m\ge n)$, and a map $a: V\mapsto \mathbb N$, where $\sum_{v\in V} a(v) = m$. Suppose that for any non-empty subset $V'\subseteq V$ the number of vertices in $W$, connected to some vertex in $V'$, is at least $\sum_{v\in V'} a(v)$. Then there exists a map $\tau: W\mapsto V$ such that $\forall w\in W$ the pair $(w, \tau(w))$ is an edge and 
$$
\forall v\in V\quad |\tau^{-1}(v)| = a(v).
$$
\end{lem}

This lemma is reduced to the ordinary Hall theorem, if every vertex $v\in V$ is split into $a(v)$ new vertices.

\begin{proof}[Proof of Theorem~\ref{antipcover3}]
Similar to the previous proofs, let us pass to the partition of unity $\phi_{i j} : B(X)\times B(Y)\to\mathbb R^+$.

Consider the functions $f_i = \sum_{j\in[m+2]} \phi_{i j}$ and $g_j = \sum_{i\in[n+2]} \phi_{i j}$. Similar to the proof of Theorem~\ref{antipcover}, the maps $f$ and $g$ can be considered as maps of $B(X)\times B(Y)$ to $n+1$ and $m+1$-dimensional simplexes respectively. Denote the inclusions
$$
i_X : B(X)\to B(X)\times B(Y)\quad i_X(x) = x\times y_0
$$
and
$$
i_Y : B(Y)\to B(X)\times B(Y)\quad i_Y(y) = x_0\times y.
$$
From the proof of Theorem~\ref{equimaps} it follows that the maps
$$
i_X^*\circ f^* : H^{n+1}(\Delta^{n+1}, \partial\Delta^{n+1}, Z_2)\to H^{n+1}(B(X), X, Z_2))
$$
and
$$
i_Y^*\circ g^* : H^{m+1}(\Delta^{m+1}, \partial\Delta^{m+1}, Z_2)\to H^{m+1}(B(Y), Y, Z_2))
$$ 
are nontrivial, and by the K\"unneth formula the following map
\begin{multline*}
(f\times g)^* : H^{n+m+2} (\Delta^{n+1}\times\Delta^{m+1}, \partial\left(\Delta^{n+1}\times\Delta^{m+1}\right), Z_2) \to\\ \to H^{n+m+2}(B(X)\times B(Y), X\times B(Y)\cup B(X)\times Y, Z_2)
\end{multline*}
is nontrivial, and therefore $f\times g$ is surjective. 

Consider the preimage of
$$
\left(\frac{a_1}{m+2}, \ldots, \frac{a_{n+2}}{m+2}\right)
\times\left(\frac{1}{m+2}, \ldots, \frac{1}{m+2}\right)\in \Delta^{n+1}\times\Delta^{m+1},
$$
denote it $p$. The matrix $\{\phi_{i j}(p)\}$ (considered as the bipartite graph incidence matrix) satisfies the conditions of Lemma~\ref{hall}, that gives the required map $\tau$, compare~\cite{kar2006}.
\end{proof}

\section{Geometry and topology of the space of $k$-flats}

Let us describe the space of all $k$-flats in $\mathbb R^n$. For any $k$-flat $\alpha$ there is a unique $(n-k)$-dimensional linear subspace of $\mathbb R^n$, orthogonal to $\alpha$, denote it $g(\alpha)$, and a unique intersection point $\alpha\cap g(\alpha)$. Thus the space of $k$-flats is parameterized by the total space $\gamma_n^{n-k}$ of the canonical vector bundle $\gamma_n^k\to G_n^{n-k}$. In the sequel we identify the space of $k$-flats with $\gamma_n^{n-k}$, thus introducing the topology on the space of $k$-flats.

In any vector bundle the group $Z_2$ acts by the fiber-wise map $x\mapsto -x$. Let us calculate the index of the space of spheres $S(\gamma_n^{n-k})$ under this action.

\begin{thm}
\label{indsg} $\hind S(\gamma_n^{n-k}) = n-1$.
\end{thm}

\begin{proof}
The space $S(\gamma_n^{n-k})$ can be viewed as the space of pairs $(n, L)$, where $n$ is a unit vector, and $L$ is a $k$-dimensional linear subspace of $\mathbb R^n$, orthogonal to $n$. Hence there is a natural equivariant map from $S(\gamma_n^{n-k})$ to the $(n-1)$-dimensional sphere $(n, L)\mapsto n$, and by Lemma~\ref{index-bu} $\hind S(\gamma_n^{n-k})\le n-1$.

Let us calculate the cohomology $H_G^*(S(\gamma_n^{n-k}), Z_2)$ as the cohomology of the quotient space $G_n^{1,k} = S(\gamma_n^{n-k})/Z_2$. This space is identified with the set of pairs $(l, L)$, where $l$ is one-dimensional subspace $\mathbb R^n$, $L$ is $k$-dimensional subspace, and $l\perp L$.

The map $\pi : (l, L)\mapsto l$ sends $G_n^{1,k}$ to $\mathbb RP^{n-1}$. it is easy to check that the one-dimensional generator $w$ of $H^*(\mathbb RP^{n-1}, Z_2)$ is mapped under $\pi^*$ to the element $w\in H^*(G_n^{1,k}, Z_2)$, let us prove that $w^{n-1}\neq 0$. Thus the space $G_n^{1,k}$ can be viewed as the space of $k$-dimensional subspaces in the fibers of the complementary canonical bundle $\eta\to \mathbb RP^{n-1}$. 

Note that the flag bundle $F(\eta)$ of the vector bundle $\eta$ is identified with the flag manifold $F(\mathbb R^n)$, and the natural projection $\pi_F : F(\eta)\to \mathbb RP^{n-1}$ 
maps a flag
$$
0\subset L_1\subset L_2\subset\dots\subset L_n = \mathbb R^n
$$
to the line $L_1$, considered as an element of $\mathbb RP^{n-1}$.
 
Let the map $\rho : F(\eta)\to G_n^{1,k}$ map a flag $L_1\subset L_2\subset\dots\subset L_n$ to the pair, consisting of $L_1$ and the orthogonal complement to $L_1$ in $L_{k+1}$. Evidently $\pi_F = \pi\circ\rho$. The map $\pi_F^*$ is known to give the injective map on the cohomology $\mod 2$. Hence the map $\pi^*$ is injective, and $w^{n-1}\neq 0$ in the cohomology of $G_n^{1,k}$.
\end{proof}

Under the notation of the previous section we formulate the lemma.

\begin{lem}
\label{bbmap} There exists a map from $B(S(\gamma_n^{n-k}))$ to the space of balls  $B(\gamma_n^{n-k})$, identical on $S(\gamma_n^{n-k})$.
\end{lem}

\begin{proof}
Let a pair $(s, t)$ represent some element $B(S(\gamma_n^{n-k}))$, let us map it to the combination $(1-t)s - ts\in B(\gamma_n^{n-k})$. The pair $(-s, 1-t)$ is mapped to the same point, thus the map $B(S(\gamma_n^{n-k}))\to B(\gamma_n^{n-k})$ is defined.
\end{proof}

Now we can deduce a theorem.

\begin{thm}
\label{bgcover} Let $S(\gamma_n^{n-k})$ be covered by a family of closed sets $\mathcal F = \{U_1,U_2,\ldots,U_{n+1}\}$. Suppose that none of $U_i$ contains a pair of antipodal points. Then for any partition of the family $\{U_i\}$ into non-empty subfamilies $\mathcal F_1$ and $\mathcal F_2$ there exists a point $c\in S(\gamma_n^{n-k})$ such that $c\in\bigcap\mathcal F_1$ and $\sigma(c)\in\bigcap\mathcal F_2$. 

Moreover, if the covering $\mathcal F$ is induced by some closed covering $\mathcal G = \{V_1,V_2,\ldots,V_{n+1}\}$ of $B(\gamma_n^{n-k})$, then $\mathcal G$ has a common point.
\end{thm}

\begin{proof}
The first claim follows from Theorems~\ref{antipcover} and \ref{indsg}, to prove the second claim we use Lemma~\ref{bbmap} to obtain a covering of $B(\gamma_n^{n-k})$ from the covering of $B(S(\gamma_n^{n-k}))$.
\end{proof}

Let us consider the oriented Grassmannian ${G_n^k}^+$. It has a natural action of $Z_2$ by the change of the orientation. Note that ${G_n^k}^+\sim{G_n^{n-k}}^+$, hence it is sufficient to consider the case $2k\le n$ to calculate the index of this action. The following theorem summarizes the data on the index of the oriented Grassmannian from the papers~\cite{hil1980A,hil1980B}. 

\begin{thm}
\label{orgindex}
Let $2k\le n$, and let $2^s$ be the minimal power of two, satisfying $2^s\ge n$.

1) If $k = 1$, then $\hind {G_n^1}^+ = \hind S^{n-1} = n-1$;

2) If $k = 2$, then $\hind {G_n^k}^+ = 2^s-2$;

3) If $k > 2$, then in the case $n=2k=2^s$ we have $2^{s-1}\le \hind {G_n^k}^+\le 2^s-1$;
and $2^s-2\le \hind {G_n^k}^+\le 2^s-1$ in other cases.

In all cases $\hind {G_n^k}^+ \ge n-k$, the equality holds for $k=1$, $k=2$ and $n=2^s$.
\end{thm}

\section{Proofs for the theorems on the canonical bundle}

\begin{proof}[Proof of Theorem~\ref{hyperpl2}]
Consider the space of balls $B(\gamma_n^1)$, we can choose the balls large enough so that $B(\gamma_n^1)$ contains all the sets $V_i$.

Let us define the subsets $U_i$ of the space $S(\gamma_n^1)$ as follows. Take some $s\in S(\gamma_n^1)$, lying in the fiber $L$. Choose the farthest from $s$ point on each of the segments $V_i\cap L$, denote it $f_i(s)$. These points depend on $s$ continuously. Now denote the nearest to $s$ point of these point by $f(s)$, it also depends continuously on $s$. 

Now put
$$
U_i = \{s\in S(\gamma_n^1) : f(s) = f_i(s)\}.
$$
These sets are closed. If some $U_i$ contains an antipodal pair $s, s'\in S(\gamma_n^1)$, then the segment $V_i\cap L$ satisfies the first alternative.

Otherwise we apply Theorem~\ref{bgcover}. For any partition of the index set $[n+1]=I_1\cup I_2$ there is a pair of antipodal points $s, s'\in S(\gamma_n^1)$ such that
$$
s\in \cap_{i\in I_1} U_i, \quad s'\in \cap_{i\in I_2} U_i.
$$
Consider the families of segments $\mathcal A = \{V_i\cap L\}_{i\in I_1}$ and $\mathcal B = \{V_i\cap L\}_{i\in I_2}$, without loss of generality assume that $s$ is to the left of $s'$. In this case all right ends of segments of $\mathcal A$ coincide in $a$, all left ends of segments of $\mathcal B$ coincide in $b$, and either $a$ is to the left of $b$, or all the segments of $\mathcal A\cup\mathcal B$ contain $[ba]$, i.e. the families of segments are equalized.
\end{proof}

\begin{proof}[Proof of Theorem~\ref{kpl1}]
Denote
$$
U_i = \{x\in\gamma_n^k : x\ \text{is in the fiber}\ L,\ \dist(x,
V_i\cap L) = \min_{j=1,\ldots,n+1}\dist(x, V_j\cap L)\}.
$$
The set $V_i\cap L$ depends continuously on $L$, hence the sets $U_i$ are closed. Denote the bundle of spheres of size $R$ in $\gamma_n^k$ by $S(\gamma_n^k)$. Let us show that for large enough $R$ the sets $U_i\cap S(\gamma_n^k)$ do not contain antipodal pairs. 

Assume the contrary. Let the set $U_i\cap S(\gamma_n^k)$ contain a pair of antipodal points $R_mx_m$ and $-R_mx_m$ for some sequence of radii $R_m\to+\infty$. It means that the closest to $R_mx_m$ and $-R_mx_m$ points of $\bigcup V_i$ belong to the same set $V_i$, denote these points by $y_m$ and $z_m$. From the compactness considerations we assume that the points  $x_m, y_m, z_m$ tend to some points $x, y, z$ in $L$. 

It is easy to see that $\conv \bigcup_{i=1}^{n+1}\{V_i\cap L\}$ contains two points $y,z$, that belong to the same $V_i$. Besides, the spheres with centers $R_mx_m, -R_mx_m$ and radii $|R_mx_m-y_m|, |-R_mx_m-z_m|$, tend to some disjoint support half-spaces for $\bigcup_{i=1}^{n+1}\{V_i\cap L\}$, containing $y$ and $z$ respectively. This is a contradiction with the non-antipodality of $\{V_i\cap L\}$. 

Hence, for large enough $R$ the sets $U_i\cap S(\gamma_n^k)$ do not contain antipodal pairs. 

Applying Theorem~\ref{bgcover} we find a common point for the family $\{U_i\}$, that is exactly what we need.
\end{proof}

\begin{proof}[Proof of Theorem~\ref{kpl2}]
Let us apply Theorem~\ref{kpl1} and find $x\in L$, equidistant from $V_i\cap L$. 

Suppose that the distance is positive. Let the set of closest to $x$ points of $\bigcup_{i=1}^{n+1} V_i\cap L$ be $K$. Obviously, $K$ intersects all of $V_i$, since $K$ and $\bigcup_{i=1}^{n+1} V_i\cap L$ are convex, then $K$ is contained in some support half-space $H$ for $\bigcup_{i=1}^{n+1} V_i\cap L$. But the opposite to $H$ support half-space cannot intersect $V_i$ from the non-antipodality condition on $\{V_i\cap L\}$, so it cannot be a support half-space for $\bigcup_{i=1}^{n+1} V_i\cap L$. That is a contradiction.
\end{proof}

\begin{proof}[Proof of Theorem~\ref{polsection}]
Consider the space $S(\gamma_n^k)$ and define its closed subspaces
$$
U_i = \{(n, L) : L\in G_n^k,\ n\in S(L),\ (n, s_i(L)) = \max_{j\in[m]} (n, s_j(L))\}.
$$
Since the interiors of $P(L)$ are non-empty, the sets $U_i$ does not contain antipodal pairs. Now Theorems~\ref{indsg} and \ref{antipcover2} imply this theorem directly.
\end{proof}

\section{Partitioning measures by hyperplanes}
\label{mespart}

Let us formulate the generalization of the theorem from~\cite{bi1990,cgppsw1994,klh1997}), that claims that any $n+1$ convex compact sets in $\mathbb R^n$ can either be intersected by a hyperplane; or any two non-empty disjoint subfamilies of this family can be separated by a hyperplane. We need some definition about measures, see the book~\cite{shg1977} for the general treatment of measures.

\begin{defn}
A measure $\mu$ on $\mathbb R^n$ is called \emph{probabilistic} if $\mu(\mathbb R^n) = 1$.
\end{defn}

We are going to consider such measures that the measure of a half-space depends continuously on the half-space. We call such measures continuous. For a measure to be continuous in this sense it is sufficient that the measure is absolutely continuous in the common sense.

\begin{defn}
A pair of a continuous probabilistic measure with compact support $\mu$ and a number $\varepsilon\in [0, 1/2)$ is called a \emph{measure with deviation}. If we consider several measures $\mu_i$ with deviation, the deviation of each measure is denoted $\varepsilon(\mu_i)$.
\end{defn}

\begin{defn}
Let $\mu$ be a measure with deviation in $\mathbb R^n$. A hyperplane $h$ \emph{(reliably) intersects} the measure $\mu$, if $h$ partitions $\mathbb R^n$ into half-spaces $H_1$ and $H_2$, and
$$
\mu(H_1),\mu(H_2)\ge \varepsilon(\mu).
$$
\end{defn}

\begin{defn}
Let $\mu$ be a measure with deviation in $\mathbb R^n$. A half-space $H$ \emph{(almost) contains} the measure $\mu$, if
$$
\mu(H) > 1 - \varepsilon(\mu).
$$
\end{defn}

\begin{defn}
Let $\mathcal M_1$ and $\mathcal M_2$ be two families of measures with deviations in $\mathbb R^n$. A hyperplane  $h$ \emph{(almost) separates} the families $\mathcal M_1$ and $\mathcal M_2$, if $h$ partitions $\mathbb R^n$ into half-spaces $H_1$ and $H_2$, and any $\mu\in\mathcal M_1$ is almost contained in $H_1$, and any $\mu\in\mathcal M_2$ is almost contained in $H_2$.
\end{defn}

Corollary~\ref{hyperpl1} implies the following claim.

\begin{cor}
\label{measuresep} 
Let $\mathcal M$ be a family of $n+1$ measures with deviation in $\mathbb R^n$. Then either there exists a hyperplane that reliably intersects all the measures of $\mathcal M$; or for any partition of $\mathcal M$ into non-empty $\mathcal M_1$ and $\mathcal M_2$ there exists a hyperplane, that almost separates $\mathcal M_1$ and $\mathcal M_2$.
\end{cor}

\begin{proof}
Put $\mathcal M = \{\mu_1,\mu_2,\ldots,\mu_{n+1}\}$, and denote $V_i$ the set of hyperplanes that reliably intersect $\mu_i$. Now applying Corollary~\ref{hyperpl1}, we obtain the required alternative.
\end{proof}

We are going to generalize some results of~\cite{bhj2007,breu2009}.

\begin{defn}
Consider a family of $n$ measures with deviation $\{\mu_i\}_{i=1}^n$ in $\mathbb R^n$. Denote by $X\subseteq\gamma_n^1$ the set of hyperplanes, that reliably intersect all these measures, consider the natural projection $p : X\to G_n^1$. The family of measures $\{\mu_i\}_{i=1}^n$ is \emph{flat}, if the projection $p$ can be lifted to the universal covering $\pi :S^{n-1}\to G_n^1$, i.e. $p =\pi\circ\tilde p$ for some continuous $\tilde p$.
\end{defn}

\begin{thm}
\label{measurepart} 
Consider a flat family of measures with deviation $\{\mu_i\}_{i=1}^n$ in $\mathbb R^n$ and the numbers $\{\alpha_i\}$ such that for any $i\in[n]$ either $\alpha_i=\varepsilon(\mu_i)$, or $\alpha_i=1-\varepsilon(\mu_i)$. Then there exists a half-space $H\subset\mathbb R^n$ such that for all $i=1,\ldots, n$
$$
\mu_i(H) = \alpha_i.
$$
\end{thm}

The flatness condition on $p$ here generalizes the condition on separated supports from Theorem~1 in~\cite{bhj2007}. If the supports are separated, then the family of measures is flat, independent of deviations. In~\cite{breu2009} a similar result is proved, the separated supports condition is replaced by the following condition: there exist $n$ separated compacts $C_1, \ldots, C_n$, such that any hyperplane, that reliably intersects $\mu_i$, intersects $C_i$. This condition implies the flatness condition of Theorem~\ref{measurepart} too. 

In Theorem~\ref{measurepart} the deviations are not equal to $1/2$, but going to the limit, we can prove it when some of the deviations are $1/2$. If all the deviations are $1/2$, we obtain the ``ham sandwich'' theorem.

\begin{proof}
Denote by $V_i$ the set of hyperplanes that reliably intersect $\mu_i$. The following is similar to the proof of Theorem~\ref{hyperpl2}.

Consider the ball bundle $B(\gamma_n^1)$, take the balls large enough so that it contains all the sets $V_i$ and define the maps $f_i : S(\gamma_n^1)\to B(\gamma_n^1)$ ($i=1,\ldots, n$) and $f : S(\gamma_n^1)\to B(\gamma_n^1)$ as in the proof of Theorem~\ref{hyperpl2}. Now define close subsets $U_0,U_1,\ldots,U_n\subseteq S(\gamma_n^1)$.

Take the projection of $X$ to $G_n^1$, denote its image by $Y$. Put $Z = p^{-1}(Y)\cap S(\gamma_n^1)$. The set $Z$ is a two-fold cover of $Y$ and by the flatness condition the covering $Z\to Y$ is trivial, i.e. $Z=Z_1\cup Z_2$, where $p: Z_1\to Y$ and $p: Z_2\to Y$ are bijections.

Now put $U_0 = Z_1$ and
$$
U_i = \{s\in S(\gamma_n^1)\setminus\int U_0 : f(s) = f_i(s) \}.
$$

The set $U_0$ does not contain antipodal pairs by definition. Suppose that some $U_i$ contains an antipodal pair $s, s'\in S(\gamma_n^1)$ in the fiber $L$. In this case the segment $V_i\cap L$ is contained in all other segments $V_j\cap L$. The length of $V_i\cap L$ is positive, since $\varepsilon(\mu_i) < 1/2$. Then $p(s) = p(s')\in \int Y$, i.e. one of the points $s, s'$ is in $\int U_0$, that is a contradiction with the definition of $U_i$.

Now put $I=\{0, 1,\ldots, n\}$,
$$
I_1 = \{i=1,\ldots, n : \alpha_i = 1 - \varepsilon(\mu_i)\}
$$
and $I_2 = I\setminus I_1$. Applying Theorem~\ref{bgcover}, we find an antipodal pair $s, s'\in S(\gamma_n^1)$ such that
$$
\forall i\in I_1\ f(s) = f_i(s),\quad\forall i\in
I_2\setminus\{0\}\ f(s') = f_i(s'),\quad s'\in\bd U_0.
$$
It follows that the segments $V_i\cap L$ intersect in the unique point, which is the right end for $V_i\cap L$ ($i\in I_1$), and the left end for $V_i\cap L$ ($i\in I_2\setminus\{0\}$). It is easy to see that this point designates the required hyperplane.
\end{proof}

\section{Borsuk-Ulam type theorems for flats}
\label{buplanes}

Let us make a definition and state a corollary of Theorems~\ref{kpl1} and \ref{kpl2}.

\begin{defn}
A set $X\subseteq\mathbb R^n$ is called \emph{$l$-convex}, if its projection to any $l$-dimensional subspace of $\mathbb R^n$ is convex.
\end{defn}

\begin{cor}
\label{kpltrans} Suppose $\mathcal F$ is a non-antipodal family of $n+1$ compact sets in $\mathbb R^n$, then there exists a $k$-flat equidistant from all the sets of $\mathcal F$. If, in addition, the union $\bigcup\mathcal F$ is $(n-k)$-convex, then $\mathcal F$ has a common $k$-transversal.
\end{cor}

\begin{proof}
Let $\mathcal F = \{K_i\}_{i=1}^{n+1}$. Denote $V_i$ the set of $k$-flats, intersecting $K_i$. In this case the sets $V_i\cap L$ are projections of $K_i$ to $L$, hence they form a non-antipodal family. Applying Theorems~\ref{kpl1} or \ref{kpl2} to $V_i$ we obtain the required result.
\end{proof}

\begin{defn}
Consider two subsets $X,Y\subseteq\mathbb R^n$. The \emph{deviation} of $X$ from $Y$ is the following number
$$
\delta(X, Y) = \sup_{x\in X} \dist(x, Y).
$$
\end{defn}

\begin{cor}
\label{kpldev} Suppose $\mathcal F$ is a non-antipodal family of $n+1$ compact sets in $\mathbb R^n$, then there exists a $k$-flat $M$ such that the deviations of all the sets of $\mathcal F$ from $M$ are equal.
\end{cor}

\begin{proof}
Let $\mathcal F = \{K_i\}_{i=1}^{n+1}$. Denote
$$
V_i = \{M\in\gamma_n^{n-k} : \delta(\bigcup\mathcal F, M) =
\delta(K_i, M)\}.
$$
Similar to the proof of Theorem~\ref{kpl1} we note that for large enough radius of balls in the ball bundle $B(\gamma_n^{n-k})$, none of the sets $V_i$ contain antipodal points in $S(\gamma_n^{n-k})$. Hence there exists non-empty intersection $\bigcap_{i=1}^{n+1} V_i$.
\end{proof}

Note again, that in Corollaries~\ref{kpltrans} and \ref{kpldev} the distance can be taken in any norm with smooth unit ball.

The fact $\hind S(\gamma_n^k) = n-1$ leads to another generalization of the Borsuk-Ulam theorem. Let us give some definitions.

\begin{defn}
Let $S^{n-1}\subset\mathbb R^n$ be the unit sphere. A \emph{$k$-subsphere} is an intersection of a $k$-dimensional linear subspace $L\subseteq\mathbb R^n$ with $S^{n-1}$.
\end{defn}

\begin{defn}
Let $S^{n-1}\subset\mathbb R^n$ be the unit sphere. A \emph{$k$-half-sphere} is a half of some $k$-subsphere.
\end{defn}

\begin{thm}
\label{halfsphere}
Let $V_1,\ldots, V_n$ be open subsets of $S^{n-1}$ such that each of $V_i$ intersects any $k$-subsphere. Then there exists a $k$-half-sphere that intersects every $V_i$.
\end{thm}

\begin{proof}
The space of $k$-subspheres is parameterized by the Grassmannian $G_n^k$, the space of all  $k$-half-spheres is parameterized by the half-spaces in the fibers of $\gamma_n^k$, with boundaries containing the origin, i.e. parameterized by $S(\gamma_n^k)$. 

Denote $U_i$ the set of $k$-half-spheres disjoint with $V_i$. This set is compact and does not contain antipodal pairs. By Theorem~\ref{indsg} $\hind S(\gamma_n^k) = n-1$, and by the generalized Borsuk-Ulam theorem the sets $U_i$ cannot cover $S(\gamma_n^k)$, that is what we need.
\end{proof}

\begin{thm}
Let $V_1,\ldots, V_{n+1}$ be open subsets of $S^{n-1}$ such that each of $V_i$ intersects any  $k$-subsphere. Then either there exist a $k$-half-sphere that intersects every $V_i$; or for any partition $[n+1]=I_1\cap I_2$ into non-empty sets there exists a pair of $k$-half-spheres $H_1$ and $H_2$, being the complementary halves of one $k$-subsphere, such that $V_i\cap H_1 = \emptyset$ for any $i\in I_1$, and $V_i\cap H_2 = \emptyset$ for any $i\in I_2$.
\end{thm}

\begin{proof}
As in the previous theorem, denote $U_i$ the set of $k$-half-spheres disjoint with $V_i$. Now the claim follows from Theorems~\ref{antipcover} and \ref{indsg}.
\end{proof}

\section{Helly-type theorems for common transversals}
\label{hellytype}

Here we state several theorems, close to the Horn-Klee theorem and its generalizations from~\cite{dol2001}. V.L.~Dolnikov has some similar results (private communication) which are not published yet.

\begin{thm}
Suppose $n+1$ families of $1$-convex compact sets $\{\mathcal F_i\}_{i\in[n+1]}$ are given in $\mathbb R^n$. Let any two sets of the same family have non-empty intersection. Then one of the alternatives holds.

1) The family $\bigcup_{i\in[n+1]} \mathcal F_i$ has $n-1$-transversal (a hyperplane);

2) For any partition of the index set $[n+1]$ into non-empty $I_1$ and $I_2$ there exists a hyperplane $h$ and a set of representatives $C_i\in\mathcal F_i$ $(i\in[n+1])$ so that the sets $\{C_i\}_{i\in I_1}$ are on one side of $h$, while the sets $\{C_i\}_{i\in I_2}$ are on the other side.
\end{thm}

\begin{proof}
For any line $l\in G_n^1$ denote $\pi_l$ the orthogonal projection onto this line, and put
$$
V_i(l) = \bigcap_{C\in\mathcal F_i} \pi_l(C).
$$
These sets are nonempty, since for any $i$ any two of the segments in $\{\pi_l(C)\}_{C\in\mathcal F_i}$ have an intersection. It is clear that $V_i(l)$ depend continuously on $l$. Put $V_i = \bigcup_{l\in G_n^1} V_i(l)$.

Apply Corollary~\ref{hyperpl1} to the family $\{V_i\}$. The first alternative of Corollary~\ref{hyperpl1} obviously corresponds to the first alternative of this theorem. 

In the other case, for any partition $[n+1] = I_1\cup I_2$ there exists a hyperplane $h$, that separates $\{V_i\}_{i\in I_1}$ and $\{V_i\}_{i\in I_2}$. Consider the projection onto the line $l\perp h$, take some directions as ``left'' and ``right'' on this line. Without loss of generality we can assume that $\{V_i\}_{i\in I_1}$ are to the left of $\pi_l(h)$, and $\{V_i\}_{i\in I_2}$ are to the right of $\pi_l(h)$. The right end of the segment $V_i$ ($i\in I_1$) is a right end of some $\pi_l(C_i)$ ($C_i\in\mathcal F_i$), the left end of the segment $V_i$ ($i\in I_2$) is a left end of some $\pi_l(C_i)$ ($C_i\in\mathcal F_i$). Hence $\{C_i\}_{i\in[n+1]}$ are the required system of representatives for the partition $[n+1] = I_1\cup I_2$.
\end{proof}

\begin{thm}
\label{colhellyanalog}
Let $0<k<n$ and suppose that $n-k+1$ families $\{\mathcal F_i\}_{i\in[n-k+1]}$ of convex compact sets are given in $\mathbb R^n$. Then one of the following alternatives holds.

1) There exists a system of representatives $K_i\in\mathcal F_i$ such that $\bigcap_{i\in[n-k+1]} K_i=\emptyset$;

2) There exists $i\in[n-k+1]$ such that in $\mathcal F_i$ any $k+1$ or less sets have a $k-1$-transversal;

3) There exists a family of parallel $k$-flats $\{\alpha_i\}_{i\in[n-k+1]}$ such that for any $i\in[n-k+1]$ the flat $\alpha_i$ is a $k$-transversal for $\mathcal F_i$.

If $2k\le n$, then the third alternative is only possible in the case $k=1$ or $k=2$ and $n=2^l$.
\end{thm}

\begin{proof}
Suppose the first alternative does not hold. Take any $L\in G_n^{n-k}$ and consider the projection of all the families to $L$. By the colored Helly theorem for some $i$ the family $\pi_L(\mathcal F_i)$ has a common point, this point corresponds to some $k$-transversal to $\mathcal F_i$, orthogonal to $L$. Denote
$$
U_i = \{L\in {G_n^k}^+ : \bigcap \pi_L(\mathcal F_i) \neq\emptyset\}.
$$

Suppose that the alternative (2) fails. Consider the subfamily (its cardinality should be exactly $k+1$) $K_1, K_2,\ldots, K_{k+1}\in\mathcal F_i$ that does not have $k-1$-transversal. For any $L\in U_i$ take the corresponding $k$-transversal $\alpha$ for the family $\mathcal F_i$, and compare the orientation on $\alpha$, given by any system of representatives $x_i\in K_i\cap\alpha$ ($i\in[k+1]$) with the orientation of $\alpha$, corresponding to $L$. All the possible systems $(x_1, \ldots, x_{k+1})$ give the same orientation, since they are never contained in a single $k-1$-flat. If the orientations coincide, assign the sign ``$+$'' to $L$ otherwise assign ``$-$'' to it. 

Thus the sets $U_i$ are mapped $Z_2$-equivariantly to $\{+1, -1\}$, and by Theorem~\ref{orgindex} and Theorem~\ref{ls} the sets $U_i$ should have a common point, that is equivalent to the alternative (3). Theorem~\ref{orgindex} tells, that it is only possible when $k=1$, or $k=2$ and $n=2^l$.
\end{proof}

In the case, when the number of families is small compared to $n$, Theorem~\ref{colhellyanalog} can be strengthened.

\begin{thm}
\label{colhellyanalog2}
Let $n = 2k + 1\ge 3$ and suppose that $2$ families $\mathcal F_1, \mathcal F_2$ of convex compact sets are given in $\mathbb R^n$. Then one of the following alternatives holds.

1) There exist two representatives $K_1\in\mathcal F_1$, $K_2\in\mathcal F_2$ such that $K_1\cap K_2=\emptyset$;

2) In some of the families $\mathcal F_i$ any $k+2$ or less sets have $k$-transversal.
\end{thm}

\begin{thm}
\label{colhellyanalog3}
Let $k>2, 2k<n+2$ and suppose that $k$ families $\mathcal F_1, \ldots,\mathcal F_k$ of convex compact sets are given in $\mathbb R^n$. Suppose that for some $m\le n-k+1$ the inequality 
$$
2^{\lceil\log_2 n\rceil}\ge k2^{\lceil\log_2(n-m)\rceil}+2
$$
holds. Then one of the following alternatives holds.

1) There is a system of representatives $K_1\in\mathcal F_1,\ldots,K_k\in\mathcal F_k$ such that $\bigcap_{i=1}^k K_i=\emptyset$;

2) In some of the families $\mathcal F_i$ any $m+1$ or less sets have $m-1$-transversal.
\end{thm}

The inequality in the statement of Theorem~\ref{colhellyanalog3} looks quite complicated, but it is true, for example, in the case $n\ge k(2n-2m-1)+2$.

We are going to use the following lemma (Lemma~5.4 from~\cite{volsce2005}).

\begin{lem}
\label{hind-union}
For compact $Z_2$-invariant subsets $X$ and $Y$ of some free $Z_2$-space $Z$ the following inequality holds
$$
\hind \left(X\cup Y\right) \le \hind X+\hind Y+1.
$$ 
\end{lem}

The following lemma generalizes the reasoning in the proof of Theorem~\ref{colhellyanalog}.

\begin{lem}
\label{transmap}
Let $k+1\le m \le n-1$ and suppose that the family $\mathcal F= \{K_1,K_2,\ldots,K_{k+1}\}$ of convex compact sets in $\mathbb R^n$ has no $k-1$-transversal. Then the set of oriented  $m$-transversals for $\mathcal F$ can be $Z_2$-mapped to $G_{n-k}^{m-k+}$.
\end{lem}

\begin{proof}
Define a vector bundle $\eta \to K_1\times\dots\times K_{k+1}$ as follows. For any system of representatives $(x_1,\ldots,x_{k+1})\in K_1\times\dots\times K_{k+1}$ the affine hull $L(x_1,\ldots, x_{k+1})$ has the dimension $k$, otherwise $\mathcal F$ would have a $k-1$-transversal. The quotient space $M(x_1,\ldots, x_{k+1}) =\mathbb R^n / L(x_1,\ldots, x_{k+1})$ is an $n-k$-dimensional vector space, and together these vector spaces form the bundle $\eta$.

The space $K_1\times\dots\times K_{k+1}$ is contractible, hence any vector bundle over it is trivial. Fix some isomorphism of vector bundles
$$
\phi : \eta \to \mathbb R^{n-k}\times K_1\times\dots\times K_{k+1},
$$
and consider its composition with the projection to the first factor
$$
\psi : \eta\to \mathbb R^{n-k}.
$$

Let the set of oriented $m$-transversals for $\mathcal F$ be $T\subseteq \gamma_n^{n-m+}$. For any $m$-flat $\tau\in T$ we can choose the $k+1$ points 
$$
x_1(\tau)\in \tau\cap K_1\ x_2(\tau)\in\tau\cap K_2,\ \ldots,\ x_{k+1}(\tau)\in \tau\cap K_{k+1},
$$
since the sets $K_i$ are strictly convex, the points $x_i$ may be chosen to depend continuously on $\tau$.

The image of $\tau$ under the natural map $\mathbb R^n\to M(x_1(t),\ldots, x_{k+1}(t))$ is an oriented $m-k$-dimensional subspace of $M(x_1(t),\ldots, x_{k+1}(t))$, and after the map $\psi : M\to \mathbb R^{n-k}$ it becomes an $m-k$-dimensional subspace in $\mathbb R^{n-k}$. Thus the required map of $T$ to $G_{n-k}^{m-k+}$ is defined.
\end{proof}

\begin{proof}[Proof of Theorem~\ref{colhellyanalog2}]
It is sufficient to prove the theorem for strictly convex compact sets, from the compactness considerations.

Denote the set of oriented hyperplane transversals for $\mathcal F_i$ by $Y_i\subseteq\gamma_n^{1+}$. Denote the natural projection of this set to $G_n^{1+} = S^{n-1}$ by $X_i=\pi_\gamma(Y_i)$.

If the first alternative fails, then $X_1\cup X_2 = S^{n-1}$ and by Lemma~\ref{hind-union} for one of $X_i$ we have
$$
\hind X_i\ge k.
$$

Now assume the contrary: there exist $k+2$ sets $K_1,\ldots, K_{k+2}$ in $\mathcal F_i$ without $k$-transversal. 

By Lemma~\ref{transmap} the set $Y_i$ can be mapped equivariantly to $G_k^{1+} = S^{k-1}$ by some map $f_i$. The natural projection $\pi_\gamma:Y_i\to X_i$ has segments as fibers, thus this map is an equivariant homotopy equivalence. Hence $\hind X_i\le k-1$ and we obtain a contradiction with Lemma~\ref{index-bu}.
\end{proof}

\begin{proof}[Proof of Theorem~\ref{colhellyanalog3}]
Again, consider the sets to be strictly convex. 

Denote by $Y_i\subseteq{\gamma_n^{k-1}}^+$ the set of $n-k+1$-transversals for $\mathcal F_i$, $X_i=\pi_\gamma(Y_i)\subseteq {G_n^{k-1}}^+$ the corresponding set of directions. The projection $\pi_\gamma : Y_i\to X_i$ has a convex set as a fiber, and has a section $\tau_i : X_i\to Y_i$.

If the first alternative fails, by the colored Helly theorem for the projections of these families to $k-1$-dimensional subspaces, we obtain $\bigcup_{i=1}^k X_i = {G_n^{n-k+1}}^+$. 

If the second alternative fails, then by Lemma~\ref{transmap} each of $Y_i$ (and therefore $X_i$) can be equivariantly mapped to ${G_{n-m}^{n-k-m+1}}^+$. By Lemma~\ref{index-bu} and Theorem~\ref{orgindex} we obtain
$$
\hind X_i \le 2^{\lceil\log_2(n-m)\rceil} - 1.
$$
Then by Lemma~\ref{hind-union} 
$$
\hind {G_n^{n-k+1}}^+ \le k2^{\lceil\log_2(n-m)\rceil} - 1.
$$
But Theorem~\ref{orgindex} gives an estimate $\hind {G_n^{n-k+1}}^+\ge 2^{\lceil\log_2n\rceil}-2$, that leads to the contradiction.
\end{proof}

\section{Acknowledgments}

The author thanks V.L.~Dolnikov for the discussion and useful remarks on the presentation of these results.


\begin{thebibliography}{99}

\bibitem{bar1982}
I.~B\'ar\'any. A generalization of Carath\'eodory's theorem. // Discrete Math., 40, 1982, 141--152.

\bibitem{bhj2007}
I.~B\'ar\'any, A.~Hubard, J.~Jer\'onimo. Slicing convex sets and measures by a hyperplane. // Discrete and Computational Geometry, 39, 2008, 67--75.

\bibitem{bi1990}
T.~Bisztriczky. On separated families of convex bodies. // Arch. Math., 54, 1990, 193–-199.

\bibitem{bor1933}
K.~Borsuk. Drei S\"atze \"uber die $n$-dimensionale euklidische Sph\"are. // Fund. Math., 20, 1933, 177--190.

\bibitem{breu2009}
F.~Breuer. Uneven splitting of ham sandwiches. // Discrete and Computational Geometry, DOI 10.1007/s00454-009-9161-7.

\bibitem{cgppsw1994}
S.E.~Cappell, J.E.~Goodman, J.~Pach, R.~Pollack, M.~Sharir, R.~Wenger. Common tangents and common transversals. // Adv. in Math., 106, 1994, 198–-215.

\bibitem{dol1993}
V.L.~Dol'nikov. Transversals of families of sets in $\mathbb R^n$ and a connection between the Helly and Borsuk theorems (In Russian), // Sb., Math. 79(1), 1994, 93--107; translation from Mat. Sb., 184(5), 1993, 111--132.

\bibitem{dol2001}
V.L.~Dolnikov. Helly-type theorems for transversals and their applications. Doctor of Mathematics thesis. Yaroslavl' State University, 2001.

\bibitem{harary1994}
F.~Harary. Graph Theory. Reading, MA, Addison-Wesley, 1994.

\bibitem{helly1923}
E.~Helly. \"Uber Mengen konvexer K\"orper mit gemeinschaftlichen Punkten. // Jber Deutsch. Math. Verein., 32, 1923, 175--176.

\bibitem{hil1980A}
H.L.~Hiller.  On the cohomology of real grassmanians. // Trans. Amer. Math. Soc., 257(2), 1980, 521--533.

\bibitem{hil1980B}
H.L.~Hiller.  On the height of the first Stiefel-Whitney class. // Proc. Amer. Math. Soc., 79(3), 1980, 495--498.

\bibitem{horn1949}
A.~Horn. Some generalization of Helly's theorem on convex sets. // Bull. Amer. Math. Soc., 55, 1949, 923--929.

\bibitem{hsiang1975}
Wu Yi Hsiang. Cohomology theory of topological transformation groups. Berlin-Heidelberg-New-York, Springer Verlag, 1975.

\bibitem{kar2006}
R.N.~Karasev. Colored version of the Knaster-Kuratowski-Mazurkiewicz lemma (In Russian). //
Modelirovaniye i analiz informatsionnyh sistem, 13(2), 2006, 66--70.

\bibitem{kar2007}
R.N.~Karasev. Colored versions of the Sperner theorem and the KKM theorem. // 
Third Russian-German Geometry Meeting dedicated to 95th birthday of A.D. Alexandrov, Abstracts, St.-Petersburg, Russia, June 18--23, 2007, 17--18.

\bibitem{klee1951}
V.~Klee. On certain intersection properties of convex sets. // Canad. J. Math., 3, 1951, 272--275.

\bibitem{klh1997}
V.~Klee, T.~Lewis, B.~Von~Hohenbalken. Appollonius revisited: supporting spheres for sundered systems. // Discrete and Computational Geometry, 18, 1997, 385–-395.

\bibitem{mak2007}
V.V.~Makeev. Some extremal problems for vector bundles (In Russian). // St. Petersbg. Math. J., 19(2), 2008, 261--277; translation from Algebra Anal., 19(2), 2007, 131--155.

\bibitem{shg1977}
G.E.~Shilov. Integral, Measure, and Derivative: A Unified Approach. Dover Publications, 1977.

\bibitem{st1942}
A.H.~Stone, J.W.~Tukey. Generalized 'Sandwich' Theorems. // Duke Math. J., 9, 1942, 356--359.

\bibitem{ste1945}
H.~Steinhaus. Sur la division des ensembles de l'espaces par les plans et et des ensembles plans par les cercles. // Fund. Math., 33, 1945, 245--263.

\bibitem{volsce2005}
A.Yu.~Volovikov, E.V.~Shchepin. Antipodes and embeddings (In Russian). // Sb. Math., 196(1), 2005, 1--28; translations from Mat. Sb., 196(1), 2005, 3--32.

\end{thebibliography}
\end{document}